\documentclass[11pt]{article}
\usepackage{amsthm,amssymb,amsmath}
\usepackage{epsfig}
\usepackage{verbatim}
\usepackage{xcolor}
\usepackage{mathrsfs}
\usepackage[utf8]{inputenc}
\usepackage{url}
\usepackage{mathtools}

\newtheorem{theorem}{Theorem}[section]
\newtheorem{lemma}[theorem]{Lemma}
\newtheorem{corollary}[theorem]{Corollary}
\newtheorem{proposition}[theorem]{Proposition}
\newtheorem{conjecture}[theorem]{Conjecture}

\newtheorem{remark}[theorem]{Remark}
\newcommand{\PP}{\mathbb P}
\newcommand{\EE}{\mathbb E}

\makeatletter
\newcommand\footnoteref[1]{\protected@xdef\@thefnmark{\ref{#1}}\@footnotemark}
\makeatother

\title{Acyclic sets and colorings in digraphs under
{restrictions on degrees and cycle lengths}}

\author{Ararat Harutyunyan\footnote{Corresponding author. Email: \texttt{ararat.harutyunyan@lamsade.dauphine.fr}} \footnote{LAMSADE, Universit\'e Paris Dauphine - PSL, Paris, France} \and Colin McDiarmid\footnote{Department of Statistics, University of Oxford, Oxford, UK} \and Gil Puig i Surroca\footnotemark[2]}

\begin{document}

\maketitle

\begin{abstract}
Given a digraph $D$, we denote by $\vec{\alpha}(D)$ the maximum size of an acyclic set of $D$ (i.e.~a set of vertices which induces a subdigraph with no directed cycles), and by $\vec\chi(D)$ the minimum number of acyclic sets into which $V(D)$ can be partitioned. In this paper, we study $\vec\alpha(D)$ and $\vec\chi(D)$ from various perspectives,
{including restrictions on degrees and cycle lengths.} A main result
 is that, if $D$ is a random $r$-regular digon-free simple digraph
of order $n$, then $\vec{\alpha}(D) = \Theta(n \log r /r)$ with high probability. This corresponds to
a result of Spencer and Subramanian on the Erd\H os--R\'enyi random digraph model \cite{SS2008}. Along the way, we derive some related results and propose some conjectures. An example of this is an analogue of the theorem of Bondy which bounds the chromatic number of a graph by the circumference of any strong orientation.
\end{abstract}


\section{Introduction}

A graph is \emph{simple} if it has no loops or parallel edges. In this paper, our graphs will always be simple. Similarly, a digraph is \emph{simple} if it has no loops or parallel arcs. We are mostly interested in digraphs that are simple. Sometimes, we shall require a digraph to be an \emph{oriented graph}, meaning that directed cycles of length two (called \emph{digons}) are also forbidden.
Digraphs are usually denoted by $D=(V,A)$, where $V$ is the set of vertices and $A$ is the set of arcs.
A subset $S$ of vertices of a digraph $D$ is called \emph{acyclic} if the induced
subdigraph on $S$ contains no directed cycle. We denote by
$\vec{\alpha}(D)$ the maximum size of an acyclic set in $D$. 
The \emph{dichromatic number} 
$\vec{\chi}(D)$ of $D$ is the smallest integer $k$ such that
$V(D)$ can be partitioned into $k$ sets $V_1,\ldots, V_k$ where each $V_i$ is acyclic. 
Note that, equivalently, the dichromatic number is the smallest integer $k$, 
such that the vertices of $D$ can be colored
with $k$ colors so that there is no monochromatic directed cycle.
It is easy to see that, for any undirected graph $G$ and 
its bidirected digraph $D$ obtained from $G$ by replacing each edge by 
two oppositely oriented arcs, we have $\chi(G) = \vec{\chi}(D)$, where
$\chi(G)$ is the \emph{chromatic number} of $G$---the minimum number of colors
needed to color the vertices of $G$ such that no two adjacent vertices have the same color.
The dichromatic number was first introduced by Neumann-Lara~\cite{N1982}.

In recent years, there has been considerable attention devoted to this topic,
and many results have demonstrated that this digraph invariant generalizes
many results on the graph chromatic number 
(see, 
for example~\cite{hajosore, BHL2018, BFJKM2004,HM11c,
HM2011}).
Some evidence of 
this surprising relationship includes the
generalization of Gallai's classical theorem on list coloring
to digraphs in \cite{HM2011}, the extension 
of the important result of Erd\H{o}s that 
sparse graphs can have large chromatic 
number to digraphs in \cite{BFJKM2004}, 
the derivation of an analogue of a classical result due to 
Bollob\'as in \cite{HM11c}, etc. 

In the present paper, we 
study the existence of large acyclic sets in general digraphs
as well as 
bounds on the dichromatic number.
We use $n$ for the number of vertices and $m$ for the number of edges or arcs.
The following two conjectures motivated our results.
The first conjecture
is due to Aharoni, Berger and Kfir. 

\begin{conjecture}\label{conj:ABK}\emph{\cite{ABK2008}}
If $D$ is an oriented graph 
 then
\[ \vec{\alpha}(D) \geq (1+o(1))\,\frac{n^2}{m} \log_2 \frac{m}{n}.
\] 
\end{conjecture}

For an $n$-vertex tournament $D$ we have $m = n(n-1)/2$, so for large tournaments the conjecture says that
\[ \vec{\alpha}(D) \geq (2+o(1)) \log_2 n \;\; \mbox{ as } \; n \to\infty.\]
This is known up to a factor of $2$; every $n$-vertex tournament contains a transitive subtournament of order $\left\lfloor\log_2 n\right\rfloor +1$. Whether the factor $2$ can indeed be gained is a major open problem which was already discussed by Erd\H os and Moser~\cite{EM1964}.

An equivalent way of stating Conjecture~\ref{conj:ABK} is that if $D$ has average outdegree $d^+=\frac{m}{n}$, then
\[ \vec{\alpha}(D) \geq (1+o(1))\,  n \log_2 d^+ \, / d^+ . \]
We note that the meaningful interpretation of the term $o(1)$ is for $d^+$ large: if, for a digraph $D$, $\vec\alpha(D) \leq (1-\varepsilon) n \log_2 d^+ /d^+$ (which is the case, for example, for the Paley tournament on seven vertices~\cite{EM1964}, or some other specific orders~\cite{S1998}), then the disjoint union of any number of copies of $D$ will also satisfy this inequality.

On the other hand, if digons are permitted,
then clearly one cannot hope to find an acyclic set of size greater than $\Theta(n/d^+)$;
this directly follows from the fact that in a symmetric digraph an acyclic set corresponds
to an independent set of the underlying graph. The extra logarithmic factor can only appear
if one considers digraphs of digirth at least three. Interestingly, in the context of undirected graphs, a similar
phenomenon occurs with independent sets.
\begin{theorem}\emph{\cite{AKS1980}} 
Every triangle-free graph $G$ of average degree $d$ has
$\alpha(G) \geq \frac{n \ln d}{100 d}$.
\end{theorem}

Shearer~\cite{S1983} improved the constant $1/100$ to $1+o(1)$ (as $d\rightarrow\infty$).
Later, Johansson~\cite{J1996,MR2002} addressed the coloring version of the problem.
\begin{theorem}
There is a constant $c$ such that, for any triangle-free graph $G$ with maximum degree $\Delta\geq 2$, $\chi(G) \leq c \Delta/\ln \Delta$.
\end{theorem}

Also in this case, the constant has been brought to {$1+o(1)$ (as $\Delta\rightarrow\infty$)} \cite{M2019}.
For oriented graphs, Erd\H{o}s and Neumann-Lara~\cite{ErdosNL} conjecture that the analogue of Johansson's theorem
(and an extension of Conjecture~\ref{conj:ABK}) should hold: 
one can color $D$ with $O(\Delta/ \log\Delta)$ colors such that each
color class is acyclic.

The next conjecture that motivated our paper is the
following one:
\begin{conjecture}
For every tournament $H$, there is $\varepsilon_H > 0$ such that 
any $H$-free tournament $T$ on $n$ vertices satisfies $\vec{\alpha}(T) \geq n^{\varepsilon_H}$. 
\end{conjecture}

(In this paper, by an \emph{$H$-free digraph} $D$ we will mean a digraph $D$ which has no subdigraph isomorphic to $H$.) Alon, Pach and Solymosi~\cite{APS2001} proved that the conjecture above is equivalent to the Erd\H{o}s--Hajnal conjecture,
one of the central questions in
extremal graph theory.

\begin{conjecture}\label{conj:EH}{\emph{\cite{EH1989}}} For every graph $H$, there is $\varepsilon_H>0$ such that any graph $G$ on $n$ vertices which does not have $H$ as an induced subgraph satisfies $\max\{\alpha(G),\omega(G)\}\geq n^{\varepsilon_H}$.
\end{conjecture}

The structure of the paper is as follows. In Section 2, we prove a theorem on digraph coloring which is reminiscent
of a well-known theorem of Bondy. The results of this
section give an upper bound on the dichromatic number when the largest cycle of a digraph has
bounded length. 
In Section~3, we propose a strengthening of the Erd\H{o}s--Hajnal conjecture, motivating it with several special cases.
{Lastly, Section~4 concerns 
random digraphs: we show that, up to a multiplicative constant, random $r$-regular oriented graphs satisfy the Aharoni--Berger--Kfir conjecture with high probability. Moreover, we also give an upper bound, which essentially differs from the lower bound by a factor of 4.}

\section{Digraph Coloring and Bondy's Theorem}

By the \emph{circumference} of a digraph $D$, we mean the maximum length of a directed cycle (when we talk about the circumference of $D$, we assume that $D$ has a directed cycle).
Bondy proved the following classical theorem.

\begin{theorem}\label{thm:Bondy}\emph{\cite{B1976}}
 Let $G$ be a graph and $D$ an orientation of $G$ which is strongly connected.
 Suppose that $D$ has circumference $s$. Then $\chi(G) \leq s$. 
\end{theorem}

A \emph{$k$-list-assignment} $L$ to a digraph $D$ is an assignment $L:V(D)\rightarrow\mathcal P(\mathbb Z^+)$ of sets of positive integers to the vertices of $D$ such that $|L(v)|\geq k$ for every vertex $v$. $D$ is \emph{$L$-colorable} if $V(D)$ can be partitioned into acyclic sets $V_1,\ldots,V_s$ so that, for every vertex $v$, $v\in\cup_{i\in L(v)} V_i$. Let $\vec{\chi}_{\ell}(D)$ denote the \emph{list dichromatic number} of $D$, that is, the minimum $k$ such that $D$ is $L$-colorable for any $k$-list-assignment to $D$.
List colorings of digraphs were introduced in \cite{HM2011} and were later studied in
\cite{BHL2018}.
\begin{theorem} \label{thm.cycle-lengths}
Let $D$ be a simple
digraph in which all (directed) cycle lengths belong to a set $K$ with $|K| = k$. 
Then $\vec{\chi}_{\ell}(D) \leq k+1$; and indeed, given any corresponding lists of size $k+1$, 
we can find a list acyclic coloring in polynomial time.
\end{theorem}

We note that both Theorem \ref{thm:Bondy} and Theorem 
\ref{thm.cycle-lengths} are tight; indeed, $K_n$ and the bidirected complete digraph on $n$ vertices, respectively, serve as examples. 
The following 
corollary of Theorem \ref{thm.cycle-lengths} is in the spirit 
of Bondy's result (Theorem~\ref{thm:Bondy}). 
\begin{corollary}\label{cor:Bondy-like}
 Let $D$ be an oriented graph with circumference 
 $s\geq 3$. 
 Then $\vec{\chi}_{\ell}(D) \leq s-1$.
\end{corollary}

To deduce the corollary, observe that in $D$ there are at most $s-2$ distinct cycle lengths.
The theorem follows from the next three easy lemmas. A digraph $D$ is $k$-\emph{degenerate} 
if in each subdigraph (including $D$ itself) there is a vertex with indegree or outdegree at most $k$. 
A $k$-\emph{degeneracy order} for $D$ is a listing of the vertices as $v_1,\ldots,v_n$ such that 
for each $j=2,\ldots,n$ vertex $v_j$ 
has indegree or outdegree at most $k$ in the subdigraph of $D$ induced on $\{v_1,\ldots,v_j\}$.

\begin{lemma}
Let $K$ be a set of $k$ positive integers, 
and let $D$ be a simple digraph in which all (directed) cycle lengths are in $K$. Then $D$ is $k$-degenerate.
\end{lemma}

\begin{proof}
Let $D'$ be a subdigraph of $D$.
Let $P = (v_1,v_2, v_3,\ldots)$ be a longest (directed) path in~$D'$.  If 
$v_1$ has an in-neighbor $w$ in $D'$ then $w$ must lie on $P$, and $w$ can only be
a vertex $v_j$ for some $j \in K$.  Thus $v_1$ has indegree at most~$k$, completing the proof.
\end{proof}

\begin{lemma}\label{lem:degeneracy_order}
Let the digraph $D$ be $k$-degenerate.  Then we can find a $k$-degeneracy order for $D$ in polynomial time.
\end{lemma}

\begin{proof}
Compute the indegree and outdegree of each vertex.
Start with an empty list.  Repeatedly, choose a vertex $v$ with indegree or outdegree at most~$k$,
put $v$ at the start of the current list, and delete $v$ from~$D$ and update the remaining indegrees and outdegrees.
\end{proof}

\begin{lemma}\label{lem:colouring_degenerate}
Let the loop-free digraph $D$ 
have a given $k$-degeneracy order $v_1,\ldots,v_n$.  Then given any lists of size $k\!+\!1$, we can read off a list acyclic coloring $f$ in polynomial time.
\end{lemma}

\begin{proof}
For each $j=1,\ldots,n$ proceed as follows. 
Let $A_j$ be a smaller of 
the two sets $N^{-}(v_j) \cap \{v_1,\ldots,v_{j-1}\}$ and $N^{+}(v_j) \cap \{v_1,\ldots,v_{j-1}\}$, so $|A_j| \leq k$; and set $f(v_j)$ to be any color in $L(v_j) \setminus \{f(w): w \in A_j\}$.
No (directed) cycle $C$ is monochromatic: for if $j$ is the largest index such that $v_j$ is in $C$, then
$C$ must contain a vertex $w \in A_j$ and $f(v_j) \neq f(w)$.
\end{proof}
This completes the proof of 
Theorem~\ref{thm.cycle-lengths}.
We next show that for tournaments we can do much better than Theorem~\ref{thm.cycle-lengths} and Corollary~\ref{cor:Bondy-like}.
Let us note first a simple but useful lemma.
\begin{lemma}
For every digraph $D$, $\vec{\chi}(D)$ (resp.~$\vec{\chi}_{\ell}(D)$) equals the maximum value of $\vec{\chi}(D')$ (resp.~$\vec{\chi}_{\ell}(D')$) over the strongly connected components $D'$ of $D$.
\end{lemma}
\begin{proof}
Let $t$ be the maximum value of $\vec{\chi}(D')$ over the strongly connected components $D'$ of $D$. We may use colors from $[t]$ to properly color each strongly connected component.  This gives a proper coloring of $D$, since each directed cycle is contained within one of the components.  The same proof works for list colorings.
\end{proof}
For each $n \in {\mathbb N}$, let $t(n)$ be the maximum value of $\vec{\chi}_{\ell}(T)$ for $T$ ranging over all $n$-vertex tournaments. Equivalently, $t(n)$ is the maximum value of $\vec{\chi}_{\ell}(D)$ for $D$ ranging over all $n$-vertex oriented graphs. From~\cite{EM1964} and~\cite{BHL2018} we have
\begin{equation} \label{eqn.tbound}
    (1/2+o(1))\, n / \log_2 n \leq t(n) 
\leq (1+o(1))\, n / \log_2 n \; \mbox{ as } n \to \infty\,
\end{equation} 
(the lower bound comes from random tournaments). Recall that a strongly connected tournament has a (directed) Hamilton cycle~\cite{C1959}; and thus for a tournament the circumference equals the maximum order of a strongly connected component. Hence, by the last lemma, if $T$ is a tournament of circumference at most $s$ then $\vec{\chi}_{\ell}(T) \leq t(s)$.
By~(\ref{eqn.tbound}) we now have:
\begin{theorem} \label{thm: tournaments}
 If $T$ is a tournament of circumference at most $s$ then
 \[ \vec{\chi}_{\ell}(T) 
 \; \leq  (1+o(1))\, s / \log_2 s 
  \; \mbox{ as } s \to \infty\,.  \]
\end{theorem}

We now show that Theorem \ref{thm: tournaments} is essentially best possible (up to a constant factor).
Indeed, one may just take a random tournament on $n=s$ vertices
. It is known that
this tournament has no acyclic set of size at least $2 \log_2{n} + 2$ with positive probability (in fact, with probability tending to $1$ as $n\rightarrow\infty$~\cite{EM1964}).
The next theorem shows that we can also choose a tournament of arbitrary order.

\begin{theorem}
For all positive integers $s$ and $n$ with $s\leq n$,
there exists an $n$-vertex tournament $T$ with circumference at most $s$ such that
 $\vec{\alpha}(T) \leq \frac{4n}{s} \log_2 (2s)$, and so 
$\vec{\chi}(T) \geq \frac{s}{4 \log_2(2s)}.$ 
\end{theorem}

\begin{proof}
Let $T$ be the following tournament. 
Partition the vertices into $\lceil\frac{n}{s}\rceil$
sets, $A_1,\ldots, A_{\lceil n/s\rceil}$, each of size at most $s$. On the vertices of each $A_i$,
orient the edges so that the largest acyclic set in
$A_i$ 
has size less than $2 \log_2 s + 2$ (this is always possible since the random tournament on $s$ vertices achieves this bound with positive probability). 
For any two vertices $v_i \in A_i$ and $v_j \in A_j$, where $i < j$, put an arc from $v_i$ to $v_j$. Note that each cycle $C$ in $T$ is contained within one of the sets $A_i$, so $C$ has length at most $s$.
Moreover,
since $\vec{\alpha}(A_i)< 2 \log_2 s + 2$,
 it follows that 
\[ \vec{\alpha}(T) < \lceil n/s \rceil (2 \log_2 s + 2) < \tfrac{4n}{s}\, \log_2 (2s).\]
Finally, $\vec{\chi}(T) \geq n/\vec{\alpha}(T) > \frac{s}{4 \log_2(2s)}$. 
\end{proof}

We now consider the problem for general oriented graphs. 
We shall see in Corollary~\ref{cor.shortcycles} that we can gain a factor of 2 over the non-list version
of Corollary~\ref{cor:Bondy-like} above (since the digirth is at least 3).

\begin{theorem} \label{thm: short cycles}
Let $D$ be a simple 
digraph with circumference $s$ and digirth~$g$.
Then $\vec{\chi}(D) \leq \lceil \frac{s}{g-1} \rceil$. 
\end{theorem}

We note that a slightly weaker version of this  
theorem is known: it was proved in \cite{CG2019} that $\vec{\chi}(D) \leq \lceil \frac{s-1}{g-1} \rceil + 1$. Our bound is clearly at most this value and for many values of $s$ and $g$ it is an improvement of one. Additionally, our proof is somewhat shorter and less technical.  
To prove Theorem \ref{thm: short cycles}, we will use a result of
Bessy and Thomass\'{e}. 

First, we need some notation. Let $D$ be a strong simple digraph on vertex set $V$.
An enumeration $E= v_1,\ldots, v_n$ of $V$ is \emph{elementary equivalent}
to another enumeration $E'$ if one of the following holds:
 $E'= v_n, v_1,\ldots, v_{n-1}$,
or $E'= v_2, v_1, v_3,\ldots,v_n$
and neither $v_1v_2$ nor $v_2v_1$ is an arc in $D$. Two  enumerations $E, E'$ of $V$ are said to be \emph{equivalent} if there is a sequence $E=E_1,\ldots,E_k = E'$ such that $E_i$ and $E_{i+1}$ are elementary equivalent for each $i$. The equivalence classes of this equivalence relation are called the \emph{cyclic orders} of $D$. Given an enumeration
$E = v_1,\ldots,v_n$ we say that an arc $v_iv_j$ is a \emph{forward arc} (with respect to $E$) if $i < j$, and a \emph{backward arc} otherwise. A directed path
in $D$ is called a \emph{forward path} if all its arcs are forward arcs. The \emph{index} of directed cycle $C$, $i_E(C)$, is the number of backward arcs of $C$. Importantly, the index of 
a cycle is invariant in a given cyclic order $\mathcal{C}$. A cycle is \emph{simple} if it has index one. A cyclic order $\mathcal{C}$ is
\emph{coherent} if for every enumeration $E$ of $\mathcal{C}$ and
every backward arc $v_jv_i$ in $E$, there is a forward path from $v_i$ to $v_j$.
The next lemma is similar to Lemma~1 of~\cite{BT2007}.
\begin{lemma} \label{lem.Bessy}
  Let $D$ be a strong simple digraph and let $C$ be a directed cycle 
 of $D$.  Then $D$ has a coherent cyclic order such that $C$ is simple.
\end{lemma}
\begin{proof}
  Amongst all cyclic orders of $D$ such that $C$ is simple, pick a cyclic order $\mathcal C$ which minimizes the sum of all cycle indices. 
Then the proof of Lemma 1 of~\cite{BT2007} shows that $\mathcal C$ is coherent.
\end{proof}

\begin{lemma} \emph{\cite[Section 4]{BT2007}} \label{Bessy}
Let $D = (V,A)$ be a strong simple digraph and $C$ be a longest cycle in $D$, of length $k$.
Suppose $\mathcal{C}$ is a coherent cyclic order of $D$ such that $C$ is simple in $\mathcal{C}$. Then there is an enumeration 
$E = v_1,\ldots,v_{i_1}, v_{i_1 + 1},\ldots,v_{i_2},$ $ v_{i_2 + 1},\ldots, v_{i_k}$ of $\mathcal{C}$ such that $\{v_{i_j + 1},\ldots,v_{i_{j+1}}\}$ is a stable set for all $j = 0,\ldots,$ $k-1$ (with $i_0:=0$). 
\end{lemma}

\begin{proof}[Proof of Theorem \ref{thm: short cycles}]
Note that it is sufficient to prove the result for strong digraphs. Indeed, if $D$ has $r$ strongly connected components $D_1,\ldots,$ $D_r$, we color each $D_i$ with at most
$\lceil \frac{s}{g-1} \rceil$ colors. Since
every directed cycle
of $D$ is contained entirely in a single $D_i$, this gives a proper coloring
of $D$. 

Thus, we may assume that $D$ is strongly connected.
Let $C$ be a cycle of length $s$ in $D$.
By Lemma~\ref{lem.Bessy}, $D$ has a coherent cyclic order $\mathcal C$ such that $C$ is simple.
Now by Lemma \ref{Bessy}, there is an enumeration 
$E = v_1,\ldots,v_{i_1}, v_{i_1 + 1},\ldots,v_{i_2}, v_{i_2 + 1},\ldots, v_{i_s}$
of $\mathcal{C}$ such that $I_{j}:= \{ v_{i_{j-1} + 1},\ldots,v_{i_{j}}\}$
 is a stable set for each $j=1,\ldots,s$. 
We claim that 
$I_j \cup I_{j+1} \cup\ldots \cup I_{j'}$
is an acyclic set for each $1 \leq j \leq j' \leq \min\{j+g-2, s\}$. Indeed,
since $I_j$, $I_{j+1},\ldots, I_{j'}$ 
are all stable sets, any cycle on $I_j\cup\ldots \cup I_{j'}$ 
must use a backward
arc $v_qv_p$, with $v_q \in I_{j+r}$ and $v_p \in I_{j+k}$, where $r > k \geq 0$. Now, since 
$D$ is assumed to be of digirth $g$, there is no forward path from $v_p$ to $v_q$.
This contradicts the fact that $\mathcal{C}$ is coherent.
Thus, we can color the digraph induced by $I_j\cup\ldots \cup I_{j'}$ 
 with one color.
Coloring consecutive $(g-1)$-tuples $I_j,\ldots,I_{j+g-2}$ (and perhaps one shorter tuple) with a single color gives a coloring
with at most $\left \lceil \frac{s}{g-1} \right \rceil$ colors.
\end{proof}

\begin{corollary} \label{cor.shortcycles}
Let $D$ be an oriented graph with circumference $s$.
Then $\vec{\chi}(D) \leq \lceil \frac{s}{2} \rceil$. 
\end{corollary}
The corollary is clearly tight for $s =3$ or $s=4$, but we think that it is not optimal for large $s$ (see Theorem~\ref{thm: tournaments}).

\begin{conjecture}
If $D$ is an oriented graph with no directed cycle of length greater than $s$, 
then $\vec{\chi}(D) = O(s / \log s)$ as $s\rightarrow\infty$.
\end{conjecture}

It was proved by Neumann-Lara~\cite{N1982} that if $D$ is an oriented graph only containing odd cycles
or only containing even cycles then $\vec{\chi}(D) \leq 2$. On the other hand,
Chen, Ma and Zang  
showed the following.

\begin{theorem}\emph{\cite{CMZ2015}} 
 Let $k$ and $r$ be integers with $k\geq 2$ and $0\leq r\leq k-1$.
If a simple digraph $D$ contains no directed cycle of length $r$ modulo $k$,
then $\vec{\chi}(D) \leq k$.
 
\end{theorem}

We show that this theorem cannot be strengthened to a list coloring version in the case 0 modulo $k$.
\begin{proposition}
For all positive integers $k\geq 3$ and $t$,
there is 
an oriented graph 
$D$ such that all cycles of $D$ have 
length 0 modulo $k$, and $\vec{\chi}_{\ell}(D) > t$. 
\end{proposition}

\begin{proof}
Let $k \geq 3$ and $t\geq 1$. Set $c = k(t-1) + 1$. Consider the following
digraph $D$. Take the directed cycle $\vec{C_k}$ with vertices $v_1,\ldots,v_k$
and blow-up each vertex $v_i$ into a set $B_i$  of independent vertices of
size $\binom{c}{t}$. Now, put a complete bipartite graph between every pair $(B_i, B_{i+1})$
with edges going from $B_i$ to $B_{i+1}$ (here, $B_{k+1}$ is $B_1$). Denote this
digraph by $D$. Clearly, all cycle lengths of $D$ are multiples of $k$. 
Let $L$ be an assignment of lists of size $t$ for $D$ such that on each set $B_i$ we see 
each $t$-subset of $\{1,\ldots,c \}$. 

Now suppose that $D$ is $L$-colorable. In any such coloring, at most $t-1$ of the
$c$ colors are absent on any given $B_i$. Indeed, suppose that on some $B_i$ we do not
see at least $t$ colors, say, the colors $\{1,\ldots, t\}$. This is clearly not possible
since then the vertex in $B_i$ with the list $\{1,\ldots, t\}$ was not colored. Thus,
in total, there are at most $k(t-1)$ colors missing from all the $B_i$. Thus,
there is some color $j$ that appears on all the $B_i$. But this is a contradiction
since we obtain a directed cycle of color $j$. Thus, $D$ is not $L$-colorable and
$\vec{\chi}_{\ell}(D) > t$.
\end{proof}

\section{Strong Erd\H{o}s--Hajnal conjecture}

In this section, we propose the following strengthening of the Erd\H{o}s--Hajnal conjecture 
(see Conjectures 1.4 and 1.5).

\begin{conjecture} \label{conj: StrongEH}
For every tournament $T$ there exists some $\varepsilon_T > 0$ such that
for any $n$-vertex $T$-free simple digraph $D$, $\vec{\alpha}(D) \geq n^{\varepsilon_T}$. 

\end{conjecture}

We say that a tournament $T$ has the \emph{Strong Erd\H{o}s--Hajnal property}
if the above conjecture is satisfied for $T$.
The following theorem gives an upper bound on the conjectured constant $\varepsilon_T$.

\begin{theorem}\label{thm:ub_zeta} 
Let $T$ be a tournament with $t\geq 3$ vertices, let $\varepsilon_T>0$, and suppose that for every $n$-vertex $T$-free oriented graph we have $\vec\alpha(D)\geq n^{\varepsilon_T}$. Then $\varepsilon_T\leq 2/t$.
\end{theorem}

The proof is probabilistic.  
We say that events $A_1,A_2,\ldots$ hold \emph{with high probability (whp)} if $\PP(A_n) \to 1$ as $n \to \infty$. 
The {\em binomial random oriented graph} $D_{n,p}$ is obtained by starting with the complete graph $K_n$, choosing each of the $\binom{n}{2}$ undirected edges independently with probability $2p$, and then orienting the chosen edges independently in one of the two directions with equal probability $\frac12$.
Thus each vertex in $D_{n,p}$ has expected indegree and expected outdegree $(n-1)p$.

\begin{proof}

Consider the random 
oriented graph $D_{n,p}$ with $p= \frac12 n^{-2/t}$.
Spencer and Subramanian~\cite{SS2008}
showed that whp 
$\vec{\alpha}(D_{n,p}) \leq \frac{2 \ln (np) }{\ln 
(1-p)^{-1}}(1+o(1))$. Thus, $\vec{\alpha}(D_{n,p}) \leq 4 n^{2/t} \ln n$ whp.
Let $X_T$ count the number of copies of $T$ in $D_{n,p}$. 
Clearly, 
$$\EE[X_T] \leq \binom{n}{t} t!\, p^{\binom{t}{2}} 
\leq n^{t} n^{-t+1}/2^{t(t-1)/2} \leq  n/8. $$
Now, by Markov's inequality, $\PP(X_T \geq n/4) \leq 1/2$.
Therefore, there is an $n$-vertex oriented graph 
$D$ with number of copies of $T$ at most $n/4$ and
with $\vec{\alpha}(D) \leq 4 n^{2/t} \ln n$. We may delete a vertex
from each copy of $T$ in $D$ to obtain a $T$-free 
oriented graph $D'$ with at least $3n/4$ vertices. 
But $\vec{\alpha}(D') \leq \vec{\alpha}(D) \leq 4 n^{2/t} \ln n$, and the theorem follows.
\end{proof}

The following special case of Conjecture~\ref{conj: StrongEH} is worthy of attention.

\begin{conjecture}
The directed triangle has the Strong Erd\H os--Hajnal property.
\end{conjecture}

We do not have any means of approaching the above special case. Nevertheless, an inductive argument shows the following weaker bound.

\begin{proposition} Let $0<\delta<\sqrt{\ln 4}$. Then there exists a constant $c_{\delta}>0$ such that, for every $\vec C_3$-free digraph $D$ of order $n$, $\vec\alpha(D)\geq c_{\delta}\exp(\delta\sqrt{\ln n})$.
\end{proposition}
\begin{proof} Let $f_{\delta}(n)\coloneqq\exp(\delta\sqrt{\ln n})$. Let $n_{\delta}$ be a positive integer, that we assume to be conveniently large. By taking $c_{\delta}$ small enough, the statement holds automatically for every digraph $D$ of order $n\leq n_{\delta}$.

Let us assume that $D$ has order $n\geq n_{\delta}$, and that the statement is true for every digraph of smaller order. We can assume that $\Delta_{\min}(D)\coloneqq\max_{v\in V(D)}\min\{\textrm{indeg}(v),\textrm{outdeg}(v)\}\geq 2n/f_{\delta}(n)$; otherwise
$$\vec\alpha(D)\geq\frac{n}{\vec\chi(D)}\geq\frac{n}{\Delta_{\min}(D)+1}\geq\frac{n}{2\frac{n}{f_{\delta}(n)}+1}\geq \tfrac{1}{3}f_{\delta}(n).$$
Let $v$ be a vertex with $\min\{\textrm{indeg}(v),\textrm{outdeg}(v)\}\geq 2n/f_{\delta}(n)$. Since $D$ is $\vec C_3$-free, there is no arc from $N^+(v)$ to $N^-(v)$. In particular, the set $N^+(v)\cap N^-(v)$ spans no arc, so $\vec\alpha(D)\geq |N^+(v)\cap N^-(v)|$. Hence, we can also assume that $|N^+(v)\cap N^-(v)|\leq f_{\delta}(n)$. Therefore, if $n^+$ and $n^-$ are the orders of the subgraphs $D^+$ and $D^-$ induced by $N^+(v)\setminus N^-(v)$ and $N^-(v)\setminus N^+(v)$, then $n^+,n^-\geq 2n/f_{\delta}(n)-f_{\delta}(n)\geq n/f_{\delta}(n)$. By applying the induction hypothesis on the disjoint subgraphs $D^+$ and $D^-$, one sees that
$$\vec\alpha(D)\geq\vec\alpha(D^+)+\vec\alpha(D^-)\geq c_{\delta}f_{\delta}(n^+)+ c_{\delta}f_{\delta}(n^-)\geq 2c_{\delta}f_{\delta}\left(\frac{n}{f_{\delta}(n)}\right)\geq c_{\delta}f_{\delta}(n),$$
where the last inequality follows from the computation
\[\lim_{x\rightarrow\infty}\frac{2f_{\delta}\left(\frac{x}{f_{\delta}(x)}\right)}{f_{\delta}(x)}
=2\lim_{x\rightarrow\infty}\exp\left(\delta\sqrt{\ln x-\delta\sqrt{\ln x}}-\delta\sqrt{\ln x}\right)\]
\[=2\exp\left\{\delta\lim_{x\rightarrow\infty}x\left(\sqrt{1-\frac{\delta}{x}}-1\right)\right\}=2\exp\left(\frac{-\delta^2}{2}\right)>1.\]
\end{proof}

Interestingly, when one forbids the transitive tournament on three vertices, the proof
is straightforward.

\begin{proposition}
Let $T$ be the transitive tournament on three vertices. Then, $T$ has the 
Strong Erd\H{o}s--Hajnal property. Moreover, for any $n$-vertex 
$T$-free simple digraph $D$ with maximum (total) degree $\Delta$, $\vec\chi(D)=O(\Delta/\log\Delta)$.
\end{proposition}

\begin{proof}
Let $D$ be a $T$-free digraph.
Take an ordering of vertices, and form a graph $G_f$ from the digraph $D$
by only keeping the forward edges. We remark that since $D$ does
not contain $T$, $G_f$ is triangle-free. Note that an independent set
in $G_f$ is an acyclic set in $D$. Now, since $G_f$ is triangle-free, by Ramsey Theory (the fact that $R(3,t) \sim t^2 / \log t$, see~\cite{K1995}), it follows that
$\alpha(G_f) = \Omega(\sqrt{n \log n})$. Thus, there is an acyclic set of size $\Omega(\sqrt{n \log n})$
in $D$. The second part of the claim follows from Johansson's theorem on colorings~\cite{J1996,MR2002} applied to the triangle-free graph $G_f$.
\end{proof}

\section{Random regular digraphs}

A (multi)digraph is \emph{$r$-regular} if every vertex has exactly $r$ in-arcs and $r$ out-arcs. For random $r$-regular $n$-vertex oriented graphs, 
we prove that the size of the largest acyclic set is $\Theta(n\ln r/r)$ whp; see Theorems~\ref{thm.ubreg} and \ref{thm.lbreg2}. 
This matches 
the behavior of the binomial random oriented graph $D_{n,p}$ with $p=r/n$, where each vertex has expected indegree and expected outdegree $r(1-1/n)$. (Recall that $D_{n,p}$ was introduced immediately after Theorem~\ref{thm:ub_zeta}.)
Indeed, by a result of Spencer and Subramanian (see Corollary~1.1 in \cite{SS2008}), as $r\rightarrow\infty$,
\[\vec\alpha(D_{n,r/n})=(1+o(1))\, \tfrac{2n\ln r}{r}  \;\;\; \mbox{whp.} \]

Random regular graphs can be constructed by means of the configuration model \cite[Section 2.4]{B2001}. Let $n$ and $r$ be positive integers such that $nr$ is even. For each vertex $i\in[n]$ we create an $r$-set $G[i]$, where the sets $G[1],\ldots,G[n]$ are pairwise disjoint. 
We then put a uniformly random pairing (a \emph{configuration}) between all the elements of $\cup_{i=1}^n G[i]$. 
Let $G^{*}(n,r)$, or simply $G^*$, be the $r$-regular multigraph obtained on vertex set $[n]$, where there is an edge $ij$ for each element of $G[i]$ that is paired with an element of $G[j]$. For $r$ fixed and $n>r$, the probability that $G^*$ is simple is bounded away from $0$ by a constant not depending on $n$; moreover, every $r$-regular $n$-vertex (simple) 
graph has the same probability of appearing as $G^*$ \cite{B2001}. Let us denote by $\mathscr G^*(n,r)$ the set of all $r$-regular $n$-vertex multigraphs with vertex set $[n]$, by $\mathscr G(n,r)$ the subset of (simple) 
$r$-regular $n$-vertex graphs, and by $P_{n,r}$ the probability measure on 
$\mathscr G^*(n,r)$ associated with $G^*$.

In the directed setting, we consider the following analogue of the configuration model. Let $n$ and $r$ be positive integers.
For each vertex $i\in[n]$ we create two $r$-sets: $D^+[i]$ and $D^-[i]$, where $D^+[1],D^-[1],\ldots,D^+[n],D^-[n]$ are pairwise disjoint. 
We denote by $D^+$ and $D^-$ the unions $\cup^n_{i=1} D^+[i]$ and $\cup_{i=1}^n D^-[i]$, respectively. 
Next, 
we put a pairing between the elements of $D^+$ and $D^-$ (a \emph{directed configuration}), uniformly at random. 
Let $D^*(n,r)$, or simply $D^*$, be the $r$-regular multidigraph obtained on vertex set $[n]$,  
where there is an arc $ij$ for each element $u \in D^{+}[i]$ which is paired with some $v \in D^{-}[j]$.
Here, the different elements of a fixed $r$-set $D^+[i]$ or $D^-[i]$ play an undistinguishable role, 
so permuting them does not affect the resulting multidigraph. More precisely, these permutations generate a group that acts on the set of all directed configurations, and each orbit corresponds to (an $r$-regular $n$-vertex) multidigraph $D$.
Thus, $D$ arises from exactly $r!^{2n}\Pi_{a\in A(D)}\tfrac{1}{\mathrm{mult}(a)!}$ directed configurations, where $A(D)$ is the set of arcs of $D$ and $\mathrm{mult}(a)$ is the multiplicity of the arc $a$. In particular, every simple 
digraph has the same probability of appearing as $D^*$.

If we forget the orientations of the arcs of $D^*$, we then obtain a $2r$-regular $n$-vertex multigraph that we call $\mathrm{forg}\,D^*$.
 We denote by $Q_{n,r}$ the probability measure on 
$\mathscr G^*(n,2r)$ associated with $\mathrm{forg}\,D^*$.

\begin{remark}\label{rem:q_measures}\emph{ We can establish a link between $\mathrm{forg}\,D^*(n,r)$ and $G^*(n,2r)$ through the enumeration of Eulerian orientations. When considering orientations of multigraphs, we have to clarify whether the edges are labelled or not. Unless specified, we will make no distinction between multiple copies of the same edge. An \emph{Eulerian orientation} of a multigraph $G$ is an orientation $D$ of $G$ such that $\mathrm{indeg}_D(v)=\mathrm{outdeg}_D(v)$ for every vertex $v$. Let $E^*_{n,r}(G)$ be the number of labelled (i.e.~edges are labelled) Eulerian orientations of $G\in\mathscr G^*(n,r)$, with the convention that loops can be oriented in two ways. Then, for all positive integers $n$ and $r$,
\begin{equation} \label{eqn.forg}
\frac{E^*_{n,2r}}{\EE[E^*_{n,2r}]}=\frac{Q_{n,r}}{P_{n,2r}},
\end{equation}
where the expectation is taken on $G^*(n,2r)$.
\begin{proof}[Proof of~(\ref{eqn.forg})] Note that in the configuration model with parameters $n$, $2r$ there are $c_{n,2r}:=\tfrac{(2nr)!}{(nr)!2^{nr}}$ possible pairings, and in the directed version of the configuration model with parameters $n$, $r$ there are $d_{n,r}:=(nr)!$ possible pairings. Thus, we see that, for every $2r$-regular $n$-vertex multigraph $G$,
\[P_{n,2r}(G)=\frac{(2r)!^n}{2^{\ell(G)}c_{n,2r}}\prod_{e\in E(G)}\frac{1}{\mathrm{mult}(e)!},\]
where $\ell(G)$ is the number of loops of $G$ and $E(G)$  is its set of edges. Similarly,
\[Q_{n,r}(G)=\frac{r!^{2n}}{d_{n,r}}\sum_{D\in\mathrm{EO}(G)}\prod_{a\in A(D)}\frac{1}{\mathrm{mult}(a)!},\]
where $\mathrm{EO}(G)$ is the set of Eulerian orientations of $G$ and $A(D)$ is the set of arcs of $D$. On the other hand,
\[E^*_{n,2r}(G)=2^{\ell(G)}\sum_{D\in\mathrm{EO}(G)}\prod_{e\in E'(G)}\binom{\mathrm{mult}_G(e)}{\mathrm{mult}_D(e^+)},\]
where $E'(G)$ is the set of non-loop edges of $G$ and, for each $e\in E'(G)$, $e^+$ is a fixed orientation of $e$ (notice that the previous expression is independent of this choice). The claim follows from the fact that 
\[\EE[E^*_{n,2r}]=\frac{2^{nr}\binom{2r}{r}^n}{\binom{2nr}{nr}}\]
(see the proof of Theorem 3.47 in~\cite{C2010}: there, $E^*_{n,2r}$ is defined in an alternative way).
\end{proof}
}
\end{remark}

Parallel to the undirected case, $D^*$ is an oriented 
graph with probability at least a positive constant. We could not find this result in the literature, so we prove it here for completeness. (In contrast, the probability that $D^*$ is simple has been studied; see for instance~\cite[Proposition 4.2]{CO-C2013}.)

\begin{lemma}\label{lem:constant_prob}
Fix a positive integer $r$ and, for every positive integer $i$, let $\mu_i=\tfrac{(2r-1)^i+1}{2i}$. Then,
$$\lim_{n\rightarrow\infty}\PP(D^*(n,r)\emph{\textrm{ is an oriented (simple) graph}})=e^{-\mu_1-\mu_2}.$$
\end{lemma}

\begin{proof} Given integers $k$, $j$, let us denote $k(k-1)\ldots(k-j+1)$ by $(k)_j$. For every positive integer $i$, let $X_{i,n}$ be the random variable counting the number of cycles of length $i$ in $G^*(n,2r)$. It is shown in~\cite[Lemma 3.51]{C2010} that
\[\lim_{n\rightarrow\infty}\frac{\EE[E^*_{n,2r}(X_{1,n})_{j_1}\ldots(X_{k,n})_{j_k}]}{\EE [E^*_{n,2r}]}=\prod_{i=1}^k \mu_i^{j_i}\]
for any set of non-negative integers $j_1,\ldots,j_k$, where the expectation is taken on $G^*(n,2r)$. By Remark~\ref{rem:q_measures},
\[\frac{\EE[E^*_{n,2r}(X_{1,n})_{j_1}\ldots(X_{k,n})_{j_k}]}{\EE [E^*_{n,2r}]}=\EE\left[\tfrac{Q_{n,r}}{P_{n,2r}}(X_{1,n})_{j_1}\ldots(X_{k,n})_{j_k}\right]=\]
\[\sum_{G\in\mathscr G^*(n,2r)}\!\left(\tfrac{Q_{n,r}}{P_{n,2r}}(X_{1,n})_{j_1}\ldots(X_{k,n})_{j_k}\right)\!(G)P_{n,2r}(G)=\EE_{Q_{n,r}}[(X_{1,n})_{j_1}\ldots(X_{k,n})_{j_k}],\]
where $\EE_{Q_{n,r}}$ is the expectation taken on $\mathrm{forg}\, D^*(n,r)$. Therefore, by the method of moments (see~\cite[Theorem 6.10]{JLR}), under the measures $Q_{n,r}$, $X_{i,n}\overset{d}{\rightarrow}\tilde X_i$ as $n\rightarrow\infty$, jointly for all $i$, where $\tilde X_i\in\mathrm{Po}(\mu_i)$ are independent Poisson random variables (see also~\cite[Lemma 9.17]{JLR}). Hence,
\[\lim_{n\rightarrow\infty}\PP(D^*(n,r) \text{ is an oriented graph})\] 
\[=\lim_{n\rightarrow\infty}\PP(\mathrm{forg}\,D^*(n,r)\in\mathscr G(n,2r))=\lim_{n\rightarrow\infty}Q_{n,r}(\mathscr G(n,2r))\]
\[=\lim_{n\rightarrow\infty} Q_{n,r}(X_{1,n}=X_{2,n}=0)=e^{-\mu_1-\mu_2}.\]
\end{proof}

We say that events $A_1,A_2,\ldots$ hold \emph{with very high probability (wvhp)}
 if $\PP(A_n) = 1-e^{-\Omega(n)}$ as $n \to \infty$.

\begin{theorem} \label{thm.ubreg}
Let $r$ be a positive integer and let $D$ be a random digraph, chosen uniformly among all $r$-regular $n$-vertex 
oriented graphs 
with labelled vertices.
 Then, $\vec{\alpha}(D) \leq \frac{2\ln r+4}{r}n$ wvhp.
\end{theorem}
\begin{proof}
Let $D^*$ be the random $r$-regular $n$-vertex multidigraph obtained with the directed version of the configuration model. By Lemma~\ref{lem:constant_prob}, $D^{*}$ is an oriented graph 
(i.e., has no parallel arcs, loops, or digons) with probability bounded away from $0$ as $n\rightarrow\infty$. 
Moreover, 
every $r$-regular $n$-vertex oriented graph 
has the same probability of appearing as $D^{*}$. Thus, it suffices to prove that $\mathbb P(\vec{\alpha}(D^*) \geq \frac{2\ln r+4}{r}n ) =e^{-\Omega(n)}$ as $n\rightarrow\infty$.

Let $k$ be a positive integer and $0 < \beta < 1$
 a real number, for now both of them unspecified. Let $\ell$ be the integer divisible by $k$ in the interval $[\beta n,\beta n+k)$. Suppose that some $A \subseteq V(D^{*})$ of size $|A| = \ell$ is acyclic. Then, there is an ordering of $A$ $\sigma: \{1,\ldots, \ell\} \to A$ such that each arc in $D^{*}[A]$ is of the form $\sigma(i) \to \sigma(j)$ for some $i < j$. This implies that $A$ can be partitioned into $k$ subsets $A_1,\ldots, A_k$ in a way that

\begin{itemize}
	\item[(a)] $|A_i|= \frac{\ell}{k} $;
	\item[(b)] for each pair $1 \leq i < j \leq k$, there is no arc from any element of $A_j$ to any element of $A_i$. 
\end{itemize}

If $\vec{\alpha}(D^{*}) \geq \ell$, then for one of the $\ell$-subsets $A$ of $V(D^*)$ the above condition must hold. The number of 
ways to choose $A$ is $\binom{n}{\ell}$ and the number of ways to partition such a set $A$ into $k$ parts 
$A_1,\ldots, A_k$ is easily at most $k^{\ell}$. Thus, in total there are at most $\binom{n}{\ell} k^{\ell}\leq(\tfrac{ekn}{\ell})^{\ell}$ choices. 

Now, let us assume that we have fixed a set $A \subseteq V(D^{*})$ with $|A| = \ell$, and a partition $A_1,\ldots, A_k$ of $A$ with $|A_i|= \frac{\ell}{k}$. Without loss of generality, we may assume that $A=\{1,\ldots, \ell\}$ and that $A_1 = \{1,\ldots, \frac{\ell}{k} \},\ldots,A_k = \{(k-1)\frac{\ell}{k} + 1,\ldots, \ell \}$. We would like to compute the probability that there is no \emph{backward} arc, i.e., no arc from $A_j$ to $A_i$ for any $j > i$. Let $E_1$ be the event that there is no arc from $A_k$ to any of the $A_i$, for all $i < k$.
Clearly, \[\mathbb P(E_1) = \prod_{j=0}^{r\frac{\ell}{k}-1} \frac{rn - r(k-1)\tfrac{\ell}{k}-j}{rn-j} \leq \left(1- \frac{(k-1)\ell}{kn}\right)^{\frac{r\ell }{k}}.\]  
In general, let $E_i$ be the event that no vertex of $A_{k-i+1}$ has a backward arc.
Then
\begin{eqnarray*}
\mathbb P(E_i \,|\, E_1,\ldots,E_{i-1})
& = &
\prod_{j=0}^{r\frac{\ell}{k} - 1}\frac{rn - r (i-1)\tfrac{\ell}{k} - r (k-i) \tfrac{\ell}{k} - j}{rn - r(i-1)\tfrac{\ell}{k} -j}\\
& \leq & 
\left(1 - \frac{ (k-i)\tfrac{\ell }{k}}{n- (i-1)\tfrac{\ell}{k}} \right)^{\frac{r \ell}{k}} 
\, \leq \;\; \left(1-\frac{(k-i)\ell}{kn}\right)^{\frac{r\ell}{k}}\\
& \leq &
\exp\left(-\frac{r(k-i)\ell^2}{k^2 n}\right).
\end{eqnarray*}
Thus 
the probability that $A$ with the partition $A_1,\ldots,A_k$ satisfies (b) is at most
\[ 
 \exp\left(- \sum_{i=1}^k \frac{r(k-i)\ell^2}{k^2 n}\right) =  \exp\left(-\frac{r(1-\frac{1}{k})\ell^2}{2n}\right).\]
Hence, 
 the probability that there is an acyclic set of size $\ell$ is at most
\[\left(\frac{ekn}{\ell}\right)^{\ell}\exp\left(-\frac{r(1-\frac{1}{k})\ell^2}{2n}\right)\leq\exp\left\{\ell\left(1+\ln k-\ln\beta-\frac{\beta r}{2}\left(1-\tfrac{1}{k}\right)\right)\right\},\] 
where we used the facts that $\frac{ekn}{\ell} \leq \frac{ek}{\beta}$
and $r(1-\frac{1}{k}) \ell^2 \geq r(1-\frac{1}{k})\beta n \ell$. 

Now, we fix $\beta=\tfrac{2\ln(3r/4)+4}{r}$ and $k=\left\lceil\beta r/2\right\rceil$. Clearly, we can assume that $r\geq 2$. This implies that $\beta>\tfrac{4}{r}$, so we have the bound $k<\tfrac{\beta r}{2}+1<\tfrac{3\beta r}{4}$. Denote by $c_r:=1+\ln k-\ln\beta-\frac{\beta r}{2}\left(1-\tfrac{1}{k}\right)$ and note that $c_r<1+\ln\tfrac{3r}{4}-\tfrac{\beta r}{2}+1=0$. Moreover, note that $c_r$ is independent of $n$. Thus, for $n$ large enough, $$\mathbb P(\vec\alpha(D^*)\geq\tfrac{2\ln r+4}{r}n )\leq
\mathbb P(\vec\alpha(D^*)\geq\beta n+k)\leq\mathbb P(\vec\alpha(D^*)\geq\ell)\leq e^{c_r\ell}\leq e^{c_r \beta n},$$
which completes the proof.
\end{proof}

Unfortunately, the bound of Theorem~\ref{thm.ubreg} is meaningless for small $r$. It makes sense to push the analysis above to try to find a constant $c<1$ such that wvhp $\vec\alpha(D)\leq cn$. Note that we cannot expect that to work for $r=1$. Indeed, it is well-known that the number of cycles of the uniform random permutation $\pi\in S_n$ is concentrated around its mean $1+\tfrac{1}{2}+\tfrac{1}{3}+\ldots+\tfrac{1}{n}$~\cite[Example III.4]{FS2009}. It follows that, when $r=1$, $\vec\alpha(D^*)=n-(1+o(1))\ln n$ with probability tending to $1$ as $n\rightarrow\infty$. The next proposition shows that, for $r\geq 2$, one can already take $c=99/100$. In any case, the bounds from Theorems~\ref{thm.ubreg} and~\ref{thm.lbreg2} are still far from each other, and bringing them closer remains an open problem.

\begin{proposition}\label{prop:r2}
Let $r\geq 2$ be an integer and let $D$ be a random digraph,
chosen uniformly among all $r$-regular $n$-vertex oriented graphs with labelled vertices. Then, $\vec\alpha(D)\leq\frac{99}{100}n$ wvhp.
\end{proposition}

\begin{proof}
Going over the proof of Theorem~\ref{thm.ubreg}, we now bound the probability of the event $E_i$ conditioned on $E_1\cap\ldots\cap E_{i-1}$ as follows:
\[\mathbb P(E_i \,|\, E_1,\ldots,E_{i-1})=\prod_{j=0}^{r\frac{\ell}{k} - 1}\frac{rn +r\tfrac{\ell}{k} - r\ell - j}{rn - r(i-1)\tfrac{\ell}{k} -j} 
\leq \left(\frac{n+\tfrac{\ell}{k}-\ell}{n - (i-1)\tfrac{\ell}{k}}\right)^{\frac{r\ell}{k}}\]
\[\leq \left(\frac{\tfrac{k}{\beta}+1-k}{\tfrac{k}{\beta}+1-i}\right)^{\frac{r\ell}{k}}, \]
and so the product $\prod_{i=1}^k\mathbb P(E_i)$ is upper-bounded by
\[\left(\left(\frac{k}{\beta}+1-k\right)^k\prod_{i=1}^k\frac{1}{\frac{k}{\beta}+1-i}\right)^{\frac{r\ell}{k}}.\]
We have that
\[\sum_{i=1}^k \ln\left(\tfrac{k}{\beta}+1-i\right) \geq \int_1^{k} \ln\left(\tfrac{k}{\beta}+1-x\right)dx\]
\[=\left[-x-\left(\tfrac{k}{\beta}+1-x\right)\ln\left(\tfrac{k}{\beta}+1-x\right)\right]_{1}^{k}\]
\[= 1-k-\left(\tfrac{k}{\beta}+1-k\right)\ln\left(\tfrac{k}{\beta}+1-k\right)+\tfrac{k}{\beta}\ln\tfrac{k}{\beta},\]
so $\mathbb P(\vec\alpha(D^*)\geq\ell)$ is at most
\[\exp\left\{\ell\left(1+r\left(1-\tfrac{1}{k}\right)+\left(1-\tfrac{r}{\beta}\right)\ln\tfrac{k}{\beta}+r\left(\tfrac{1}{\beta}+\tfrac{1}{k}\right)\ln\left(\tfrac{k}{\beta}+1-k\right)\right)\right\}.\]
Therefore, it is enough to ask that
$$1+r\left(1-\tfrac{1}{k}\right)+\left(1-\tfrac{r}{\beta}\right)\ln\tfrac{k}{\beta}+r\left(\tfrac{1}{\beta}+\tfrac{1}{k}\right)\ln\left(\tfrac{k}{\beta}+1-k\right)<0,$$
which, for $r\geq 2$, is satisfied by $k=100$ and $\beta=99/100$.
\end{proof}

There is a natural (fast) greedy algorithm which yields an acyclic set 
in a loop-free digraph $D$. (The $U$ is for `unused'.)\\

input: \, a digraph $D$ and a total order $\preceq$ on $V(D)$

set $A=U=\emptyset$ and $W=V(D)$

while $W \neq \emptyset$ 

\hspace{.3in}
let $w$ be the first (smallest) vertex in $W$ 
 and reveal $N^+(w)$

\hspace{.3in} move $w$ into $A$ and move $N^+(w) \cap (W \backslash \{w\})$ into $U$

output: \, $A$\\

Observe that the output set $A$ is acyclic if we ignore any loops, 
since all arcs 
point `backwards' to vertices added earlier. We call $A$ the \emph{greedy acyclic set} of $D$ with respect to $\preceq$, and denote its size by $\vec{\alpha}(D,\preceq)$. 
 (If we wanted to find a large acyclic set in a general loop-free digraph, not necessarily random, we would pick $\preceq$ uniformly at random.)

\begin{theorem} \label{thm.lbreg2}
Let $\preceq$ be a total order on $[n]$, 
let $r$ be a positive integer, and let $D$ 
 be a random digraph, chosen uniformly among all $r$-regular 
oriented graphs 
on $[n]$.
  Then $\vec{\alpha}(D,\preceq)$ $\geq \tfrac15 \, n \log_2 (r+1) /r$ wvhp. Moreover, for any $\varepsilon >0$ there exists an $r_{\varepsilon}$ such that, if $r\geq r_{\varepsilon}$, then $\vec{\alpha}(D,\preceq) \geq (\tfrac12-\varepsilon) \, n \ln (r+1) /r$ wvhp.
\end{theorem}
Note that our lower bounds on $\vec{\alpha}(D,\preceq)$ match the upper bound in Theorem~\ref{thm.ubreg} on $\vec{\alpha}(D)$ up to a constant factor.
From now on, $\log$ means $\log_2$. Let $0<\alpha<1$ be fixed. To prove Theorem~\ref{thm.lbreg2} we shall use a truncated version of the above greedy algorithm, where we replace the condition `while $W \neq \emptyset$' by `while $W \neq \emptyset$ and $|A| \leq \alpha n$'. Later we shall set $\alpha=\frac15$.

\smallskip

We shall prove that, when $\alpha=\frac15$, for $D^*$ the random $n$-vertex $r$-regular multidigraph obtained with the directed version of configuration model (as in the proof of Theorem~\ref{thm.ubreg}), the algorithm yields a set $A$ with 
\begin{equation} \label{eqn.forthm}
|A| \geq \tfrac15 \, n \log(r+1)/r \;\text{ wvhp}.
\end{equation}
But by Lemma~\ref{lem:constant_prob} the probability that $D^*$ is an oriented graph 
is bounded away from 0, and all oriented graphs 
have the same probability of appearing as $D^*$, so the theorem will follow. (The proof below shows that, essentially, the bounds of Theorem~\ref{thm.lbreg2} hold also for $\vec\alpha(D^*)$, 
since the expected number of loops in $D^*$ is $r$.)

\smallskip

We use one preliminary lemma.  Let $n \geq 1$ and $0 \leq a,b \leq n$.  Let $U$ and $V$ be disjoint $n$-sets, and let $G$ be the complete bipartite graph with parts $U$ and $V$.  Let $M$ be a random perfect matching in $G$ chosen uniformly from the $n!$ perfect matchings in $G$.
Let $A \subseteq U$ with $|A|=a$ and $B \subseteq V$ with $|B|=b$. Let the random variable $X(n,A,B)$ (or less precisely $X(n,a,b)$) be the number of edges in $M$ between $A$ and $B$. Observe that $\EE[X] = ab/n$. 

If $X,Y$ are two random variables, we say that $X$ is \emph{stochastically dominated} by $Y$ if $\PP(X\geq t)\leq \PP(Y\geq t)$ for every $t$, and we denote it by $X\leq_s Y$.

\begin{lemma} \label{lem.X}
  Let $n, n' \geq 1$, let $a,b \leq n$, let $a', b' \leq n'$; and suppose that $n' \leq n$, $a' \geq a$ and $b' \geq b$.  Let $Y = X(n,a,b)$ and $Y'=X(n',a',b')$.  Then $Y \leq_s Y'$.  
\end{lemma}
\begin{proof}[Proof of Lemma~\ref{lem.X}]
It suffices to establish the following three simple claims.

\begin{equation} \label{claim1}
\mbox{If } \; a+1 \leq n \; \mbox{ then } \; X(n,a,b) \leq_s X(n,a+1,b).
\end{equation}

\begin{equation} \label{claim2}
\mbox{If } \; b+1 \leq n \; \mbox{ then } \; X(n,a,b) \leq_s X(n,a,b+1).
\end{equation}

\begin{equation} \label{claim3}
 X(n+1,a,b) \leq_s X(n,a,b).
\end{equation}
To prove~(\ref{claim1}), let $a+1 \leq n$; let $A \subseteq A' \subseteq U$ with $|A|=a, |A'|=a+1$; and let $B \subseteq V$ with $|B|=b$.
Then always $X(n,A,B) \leq X(n,A',B)$ so~(\ref{claim1}) holds.  We may prove~(\ref{claim2}) in the same way.

It remains to prove~(\ref{claim3}).  Let $U$ and $V$ be disjoint $(n+1)$-sets, let $A \subseteq U$ with $|A|=a \leq n$ and let $B \subseteq V$ with $|B|=b \leq n$.
Let $u \in U \backslash A$, and let $v$ be the random vertex in $V$ paired with $u$ in the random matching $M$.   Then for each relevant integer $i$
\begin{eqnarray*}
&& \PP(X(n+1,a,b) \geq i)\\
&=&
\PP((X(n+1,A,B) \geq i) \land (v \in B)) + \PP((X(n+1,A,B) \geq i) \land (v \not\in B)) \\
&=&
\tfrac{b}{n+1}\, \PP(X(n,a,b-1) \geq i) + \tfrac{n+1-b}{n+1}\, \PP(X(n,a,b) \geq i)\\
& \leq &
\tfrac{b}{n+1}\, \PP(X(n,a,b) \geq i) + \tfrac{n+1-b}{n+1}\, \PP(X(n,a,b) \geq i) \;\;\mbox{ by } (\ref{claim2})\\
&=&
\PP(X(n,a,b) \geq i).
\end{eqnarray*}
Now~(\ref{claim3}) follows, and the proof is complete.
\end{proof}

\begin{proof}[Proof of~(\ref{eqn.forthm}), and thus of Theorem~\ref{thm.lbreg2}]
Consider part way through a run of the algorithm, when we are about to reveal $N^+(w)$.
At this time, we know the sets $A$, $U$ and $W$ of vertices in $G$; we know all the $r$ arcs out of each vertex in $A$ (that is, we know the edges in the random matching $M$ which are incident to the points corresponding to an out-incidence of a vertex in $A$), and all these arcs go from $A$ to $A \cup U$.  The remaining $r(n-|A|)$ edges in $M$ form a uniformly random perfect matching $M'$ in the bipartite graph over the remaining points.
Let the \emph{cost} $Y$ when revealing $N^+(w)$
be the consequent reduction in the size of $W$. Then
\begin{equation} \label{eqn.cost}
Y \leq_s 1+ X(r(n-|A|),r,r|W|) \leq_s 1+ X(r(n-\lfloor\alpha n\rfloor),r,r|W|),
\end{equation}
where we use Lemma~\ref{lem.X} for the second inequality $\leq_s$.
Note that $1 \leq Y \leq r+1$, and $\EE[Y] \leq 1+ \frac{r|W|}{n(1-\alpha)}$.

Let $b\geq 1+r^{-1}$ be a constant. We assume that $n$ is large enough. Divide the runs of the algorithm into stages $s=1,2,\ldots$, where stage $s$ is when $n b^{-s+1} \geq |W| > nb^{-s}$.
Consider stage $s$, where $1 \leq s \leq \log_b (r+1)$.
Let $Y_1,Y_2,\ldots$ be the costs of the first, second,...~vertices added to $A$ in this stage, where we set $Y_i=0$ if the algorithm stops before adding the $i$th vertex or $|W|$ has decreased to at most $n b^{-s}$. If we add an $i$th vertex, then at this time $|W| \leq nb^{-s+1}$ and $|A| \leq \lfloor\alpha n\rfloor$, so by~(\ref{eqn.cost}) (conditional on all history to date)
\[ Y_i \leq_s 1+X(r(n-\lfloor\alpha n\rfloor), r,r|W|) \leq_s Z \]
where $Z \sim 1+X(r(n-\lfloor\alpha n\rfloor), r, rnb^{-s+1})$.  Further, let $Z_1,\ldots,Z_k$ be independent copies of $Z$: then
jointly $(Y_1,\ldots,Y_k) \leq_s (Z_1,\ldots,Z_k)$,
and so $\sum_{i=1}^k Y_i \leq_s \sum_{i=1}^k Z_i$.
Recall that $1 \leq Z \leq r+1$ and
$\EE[Z] = 1+\tfrac{rnb^{-s+1}}{n-\lfloor\alpha n\rfloor}\leq 1+\beta rb^{-s+1}$, where $\beta=(1-\alpha)^{-1}$.  But $rb^{-s+1} \geq \frac{br}{r+1} \geq 1$, so $\EE[Z] \leq (1+\beta) r b^{-s+1}$. 

Let $\gamma<(\beta+1)^{-1}$ and $k = \frac{\gamma n(b-1)}{br}$, and note that $(k+1)\, \EE[Z] \leq (\beta+1)\gamma n(b^{-s+1}-b^{-s})+(1+\beta)rb^{-s+1}$.  Thus $\sum_{i=1}^{\lceil k\rceil} Z_i \leq n(b^{-s+1}-b^{-s}) - (r+1)$ wvhp, by a standard inequality, see for example~\cite[Section 10.1]{MR2002}.
Hence, in this stage wvhp either we add at least $\lceil k\rceil$ vertices to $A$, or by the end of the stage we have 
$|A| \geq \alpha n$. 
This holds for each stage $s=1,\ldots,\log_b (r+1)$.  Hence after these stages, wvhp either $|A| \geq \log_b (r+1) k = \frac{\gamma (b-1) n \log_b (r+1)}{br}$ or $|A| \geq \alpha n$. 

Finally set $b=2$, $\alpha = \frac15$ (so $\beta=\frac54$), and $\gamma=\frac25$.  Then
\[ \min \{ \tfrac{\gamma (b-1) \log_b (r+1)}{br}, \alpha \} =
\min \{ \tfrac{\log(r+1)}{5 r}, \tfrac15 \} = \tfrac{\log(r+1)}{5 r}.\]
Thus after the stages above we have $|A| \geq \frac{\log(r+1)}{5 r} n$ wvhp, and we have proved~(\ref{eqn.forthm}) as required. Alternatively, if $b=1+r^{-1}$ and $\alpha,\gamma$ are chosen arbitrarily close to $0$ and $\tfrac12$, respectively, and assuming that $r$ is large enough, then
\[ \min \{ \tfrac{\gamma (b-1) \log_b (r+1)}{br}, \alpha \} =\tfrac{\gamma (b-1) \log_b (r+1)}{br}\geq(\tfrac12-\varepsilon)\tfrac{\ln(r+1)}{r}\]
for any given $\varepsilon\in\mathbb R^+$.
\end{proof}

We know that for every $r$-regular simple digraph $D$,  $\vec{\chi}(D) \leq r+1$  and so $\vec{\alpha}(D) \geq n/(r+1)$ (see Lemmas~\ref{lem:degeneracy_order} and~\ref{lem:colouring_degenerate}). The lower bound here on $\vec{\alpha}(D)$ is 
 better than that in Theorem~\ref{thm.lbreg2} for small $r$. 

Finally, note that both Theorem~\ref{thm.ubreg} and Theorem~\ref{thm.lbreg2} hold also if $D$ is chosen uniformly at random among all $r$-regular $n$-vertex simple digraphs (i.e.~allowing digons). Indeed, we may use essentially the same proofs: in the first paragraph of the proof of Theorem~\ref{thm.ubreg} we can just replace `oriented graph' by `simple digraph', and we can do the same in the paragraph following~(\ref{eqn.forthm}).

\vspace{5mm}
\textbf{Acknowledgements.} The first and third authors are supported by the ANR project 21-CE48-0012 DAGDigDec (DAGs and Digraph Decompositions). The third author is also partially supported by the MICINN project PID2022-137283NB-C22 ACoGe (Algebraic combinatorics and its connections to geometry). Some of this work was carried out during the First Armenian Workshop on Graphs, Combinatorics, Probability and their Applications to Machine Learning (2019) and the authors
are grateful to the organizers. The authors are also grateful to the referees for helpful comments.

\end{document}